\documentclass{article}

\usepackage{amsmath}
\usepackage{amsthm}
\usepackage{amsfonts}
\usepackage{amsbsy}
\usepackage{amssymb}
\newtheorem{theorem}{Theorem}
\newtheorem{proposition}{Proposition}

\newcommand{\R}{{\mathbb R}}
\newcommand{\Z}{{\mathbb Z}}

\newcommand{\set}[2]{ \left\{ #1 \ \left| \ #2 \right. \right\} }

\newcommand{\rank}{\mbox{rank}}
\newcommand{\locrank}{\mbox{loc rank}}
\newcommand{\lless}{< \! \! <}
\newcommand{\ggreater}{> \! \! >}

\title{Uniform estimates for cubic oscillatory integrals}
\author{Philip T. Gressman\footnote{Partially supported by NSF grant DMS-0653755.}}

\begin{document}
\maketitle
\begin{abstract}
This paper establishes the optimal decay rate for scalar oscillatory integrals in $n$ variables which satisfy a nondegeneracy condition on the third derivatives.  The estimates proved are stable under small linear perturbations, as encountered when computing the Fourier transform of surface-carried measures.  The main idea of the proof is to construct a nonisotropic family of balls which locally capture the scales and directions in which cancellation occurs.
\end{abstract}

The purpose of this paper is to establish decay estimates for the scalar oscillatory integral
\begin{equation}
I(\lambda,\xi) := \int e^{i \lambda ( \Phi(x) + \xi \cdot x)} \psi(x) dx \label{oscillate}
\end{equation}
(where $x \in \R^n$, $\Phi$ and $\psi$ are real-valued and $\psi$ is compactly supported in some convex domain $\Omega$) which are uniform in $\xi$, in the case when the Hessian of $\Phi$ is degenerate but has some type of first-order nondegeneracy (corresponding to a condition on the third derivatives of $\Phi$). 
The integral \eqref{oscillate} arises naturally (after rescaling $\xi$) when taking the Fourier transform of the surface measure on the graph $(x,\Phi(x)) \subset \R^{n+1}$, which is itself intimately connected to many classical and modern problems in analysis; see Stein \cite{steinha} or Bruna, Nagel, and Wainger \cite{bnw1988} for discussion and a thorough collection of references to earlier work.  More recently, the issue of stability of oscillatory integrals has been the focus of the work of Phong, Stein, and Sturm \cite{pss1999}, \cite{pss2001} and Phong and Sturm \cite{ps2000}.  In addition, stability considerations are often implicit in the vast bodies of work on Radon transforms, oscillatory and Fourier integral operators, and variations on these objects.

As in the case of most treatments of the integral \eqref{oscillate}, the method of stationary phase will be the primary tool used.  
In contrast with earlier work along these lines (for example, Var\v{c}enko \cite{varcenko1976}, Pramanik and Yang \cite{py2004} or Bruna, Nagel, and Wainger \cite{bnw1988}), the goal here is to avoid any assumptions (either on the nature of the Newton polyhedron in \cite{varcenko1976} or \cite{py2004} or on the convexity of the graph of $\Phi$ as in \cite{bnw1988}) which force $\Phi(x) + \xi \cdot x$ to have ``uniformly isolated'' critical points, since the critical points of $\Phi(x) + \xi \cdot x$ can decompose and coalesce as $\xi$ varies.  To mitigate this substantial new difficulty, the phase $\Phi$ will be assumed to have third derivatives which are nondegenerate in an appropriate sense.

When establishing uniform estimates, there is also an added difficulty that the oscillation index (i.e., the rate of decay of \eqref{oscillate} as a power of $\lambda$) is not in general an upper semicontinuous function of $\xi$ as might be hoped.  The classical example of the failure of semicontinuity is due to Var\v{c}enko \cite{varcenko1976}.  There is, however, a more pertinent example to the problem at hand:  consider the phase $\Phi(x_1,x_2,x_3,x_4) := -x_1^3 + x_1 (x_2^2 + x_3^2 + x_4^2)$.  This phase is homogeneous of degree $3$ and nondegenerate in the sense that the only critical point is at the origin.  Such integrals have been thoroughly studied; the work of Karpushkin \cite{karpushkin1995}, for example, applies to this phase and dictates that $|I(\lambda,0)| \leq C |\lambda|^{-4/3}$.  But exploiting the spherical symmetry of $\Phi$ in the second through fourth coordinates allows one to rewrite \eqref{oscillate} as a weighted oscillatory integral in the plane when $\xi = (-\epsilon,0,0,0)$ (the new phase being $- \epsilon x_1 - x_1^3 + x_1 r^2$).  For any $\epsilon > 0$, the critical point of the two-dimensional phase is nondegenerate (meaning that the Hessian matrix of $\Phi$ is invertible there) and is located away from the line $r=0$, thus one can only expect $|I(\lambda,\xi)| \leq C_\xi |\lambda|^{-1}$.

As in the work of Greenleaf, Pramanik, and Tang \cite{gpt2007} on oscillatory integral operators, there are two different approaches to estimating \eqref{oscillate}.  The first involves formulating a fairly explicit nondegeneracy condition for phases $\Phi$.  Loosely speaking, the condition is that, in the neighborhood of a critical point $x_0$, the magnitude of the gradient of $\Phi$ grows at least quadratically in the distance to that critical point {\it and} that the critical points of all linear perturbations of $\Phi$ have this property as well.  This precludes the pathologies of the phase $-x_1^3 + x_1 (x_2^2 + x_3^2 + x_4^2)$, since in this latter case the critical points of the perturbed phase need not be isolated.
The precise statement of this condition goes as follows:
suppose $\Omega \subset \R^n$ is an open, convex set, and suppose that $\Phi$ is real-valued function on $\Omega$ with bounded derivatives of all orders up through order $n+1$.
In particular, it will be assumed that there is some finite $K$ such that
for any unit vectors (with respect to the standard Euclidean metric) $w_1,w_2,w_3$, 
\[\left| (w_1 \cdot \nabla) (w_2 \cdot \nabla ) (w_3 \cdot \nabla) \Phi(x)\right| \leq K \]
for all $x \in \Omega$ (this will be referred to as the boundedness condition).  Moreover, the following nondegeneracy condition condition will be assumed:  for each $x \in \Omega$ and each finite $\mu$, let $V_{\mu,x}$ be the vector space of eigenvectors of the Hessian matrix of $\Phi$ at $x$, denoted $H_x$, with eigenvalues $\nu$ satisfying $|\nu| \leq \mu$.  The phase $\Phi$ will be said to satisfy the nondegeneracy condition when there exist constants $K'$, $M$, and $R$, so that,  for any $\mu \leq M$ and any unit vector $v \in V_{\mu,x}$, there is a unit vector $w \in V_{R \mu,x}$ such that
\begin{equation} (v \cdot \nabla) (v \cdot \nabla) (w \cdot \nabla) \Phi(y) \geq K' \label{nondegen}
\end{equation}
for all $y \in \Omega$.  Under these conditions, the following theorem holds:
\begin{theorem}
Suppose $\Phi$ satisfies the boundedness and nondegeneracy conditions.  Then for any $\psi$ compactly supported in $\Omega$, there is a constant $C$ such that
\[ \left| \int e^{i \lambda (\Phi(x) + \xi \cdot x)} \psi(x) dx \right| \leq C \lambda^{-\frac{n-k}{2} - \frac{k}{3}} \]
for all $\xi$ sufficiently small and all real $\lambda$, where $k$ is the infimum over all $x$ in the support of $\psi$ of the dimension of $V_{M,x}$. \label{oscthm}  Moreover, the exponent $- \frac{n-k}{2} - \frac{k}{3}$ is optimal in the sense that there exist phases satisfying \eqref{nondegen} and appropriate amplitude functions $\psi$ for which \eqref{oscillate} has magnitude at least equal to some constant times $\lambda^{-\frac{n-k}{2} - \frac{k}{3}}$.
\end{theorem}
The second type of result is primarily algebraic (the ``low-lying fruit'' of Greenleaf, Pramanik, and Tang \cite{gpt2007}).  The main idea behind this approach is that, in sufficiently high dimensions, the Hessian matrix of a generic cubic polynomial at a point $x \neq 0$ is ``nearly'' nondegenerate, meaning that the rank is asymptotic to the dimension $n$.  This allows for a somewhat different approach to estimating \eqref{oscillate}.   Let ${\mathfrak S}_n^3$ be the real vector space of cubic polynomials in $n$ variables (given the usual metric topology).  When the Hessian of $\Phi$ is zero at the origin, then a uniform estimate also holds generically in the following sense:
\begin{theorem}
For any dimension $n \geq 18$, there is a dense open set $U_n \subset {\mathfrak S}_n^3$ such that, for any $p \in U_n$, if $\Phi(x) - p(x)$ vanishes to fourth order (or higher) at the origin, then there is a constant $C$ such that \label{oscthm2}
\[ \left| \int e^{i \lambda (\Phi(x) + \xi \cdot x)} \psi(x) dx \right| \leq C \lambda^{-\frac{n}{3}} \]
provided $\xi$ is sufficiently small and $\psi$ is supported in a sufficiently small neighborhood of the origin.
\end{theorem}
It should be noted that similar results hold as long as the kernel of the Hessian of $\Phi$ has non-negligible dimension, which in this case means that it is larger than some fixed constant times $n^\frac{1}{2}$.  An interesting consequence of the proof of \ref{oscthm2} is that the set of ``bad'' cubics in $n$ variables actually has codimension greater than $1$ in ${\mathfrak S}_n^3$.  This is in sharp contrast with the standard results for quadratics (``bad'' quadratics are completely characterized in this context by having zero determinant).  This observation partly explains why it appears to be so difficult to explicitly characterize that set (as noted in \cite{gpt2007}).

{\bf Examples.}  An example of a phase $\Phi$ satisfying the conditions of theorem \eqref{oscthm} is given by $\Phi(x) := \sum_{i=1}^k x_i^3 + \sum_{i=k+1}^n x_i^2$.  In and of itself, this example is of limited interest (although it is readily seen to prove the optimality of theorem \ref{oscthm}), since the integral \eqref{oscillate} factors into a product of one-dimensional integrals in this case.  The novelty in this case is that the nondegeneracy condition continues to hold for a class of smooth perturbations which break the factorization (for example $\Phi(x) +  (x_1 x_2 x_3)^3$).  Other more complicated examples, like $\Phi(x,y) := x^3 - 3 xy^2$ exist as well (and, again, continue to satisfy the nondegeneracy condition under some class of smooth perturbations).  At the other end of the spectrum, theorem \eqref{oscthm2} is quite broadly applicable, but even given more detailed information about the set $U$, it quickly becomes very difficult to verify whether a given phase $\Phi$ is indeed generic or not.  The computation is in theory an explicit one, using the machinery of resultants (as appeared in the work of Greenleaf, Pramanik, and Tang \cite{gpt2007}, see the book of Gel$'$fand, Zelevinski{\u\i}, and Kapranov  \cite{gkz1994} for a complete exposition) as well as the machinery of determinantal resultants as developed by Bus\'e \cite{buse2004}.  In practice, however, it quickly becomes impossible to write down explicit examples for which \eqref{oscthm2} applies because the sum of two generic phases of lower degree $\Phi_1(x) + \Phi_2(y)$ will not necessarily be generic as a function of $x$ and $y$.

As can be expected for problems of this type, the main element of the proof of theorem \ref{oscthm} is an integration-by-parts argument.  In this case, the natural sets on which to perform the integration-by-parts are a family of nonisotropic balls which are intimately connected to the geometry of the Hessian matrix of $\Phi$ (similar to the work of Bruna, Nagel, and Wainger \cite{bnw1988}).   In the proof of various one-dimensional generalizations of the classical van der Corput lemma, two facts about degenerate phases become clear:  first, oscillatory integrals with degenerate phases have less ``total'' cancellation than integrals with nondegenerate phases.  Second, the cancellation occurring for degenerate phases happens over longer scales (i.e., it takes more ``room'' for cancellation to occur).  In higher dimensions, cancellation can take place on a variety of different length scales and in different directions.  The nonisotropic balls given in section \ref{ballsec} reflect the natural local length scales at which cancellations occur in \eqref{oscillate}.  Following that, the integration-by-parts is performed, and the general situation is reduced to integration on nonisotropic balls by constructing an appropriate partition of unity adapted to those balls.  Finally, in section \ref{gradientsec}, the results of the integration-by-parts argument are applied to the specific case of theorem \ref{oscthm}.  The key idea behind this application is an inductive decomposition of the domain into pieces on which there are large gaps in the spectra of the Hessian matrices $H_x$.  The proof of theorem \eqref{oscthm2} comes in section \ref{genericsec}.


\section{Nonisotropic ball geometry}
\label{ballsec}
To begin this section, a brief explanation of convention is in order.  Given two quantities $A$ and $B$, the expression $A \lesssim B$ will mean that there exists a constant $C$ depending only on the dimension $n$ such that $A \leq C B$.  Likewise $A \approx B$ will mean $A \lesssim B$ and $B \lesssim A$.  The expression $A \lless B$ will stand for the phrase ``there exists a sufficiently small constant $c$ depending only on dimension such that $A \leq c B$.''  The distinction between these two conventions is that $A \lless B$ will only appear as the hypothesis of an implication, while $A \lesssim B$ will only appear as the conclusion of an implication.
Finally, here and throughout, $| \cdot |$ will represent the standard Euclidean length of a vector.

Let $H_x$ be the Hessian matrix of the phase function $\Phi$ at $x$, and let $E^\mu_x$ be the spectral projection onto the eigenspace of $H_x$ with eigenvalue $\mu$.  The eigenvectors and eigenvalues of $H_x$ will be used to construct a nonisotropic family of balls.  Before this can be accomplished, it is necessary to consider the continuity properties of the spectrum itself, as expressed by the following proposition:
\begin{proposition}
For any $x,y \in \Omega$ and any real numbers $\mu_1,\mu_2$,
\begin{equation} ||E^{\mu_1}_x E^{\mu_2}_y || \leq \frac{\min\{|\mu_1|,|\mu_2|\} + K |x-y|}{\max\{|\mu_1|,|\mu_2|\}}. \label{spectrumpert}
\end{equation}
Furthermore, if $|x-y| \leq K^{-\frac{1}{3}} d^\frac{1}{3}$ and $r_1, r_2 \leq d$ then
\begin{equation}
||E^{\mu_1}_x E^{\mu_2}_y || \leq 2 d^\frac{1}{6} r_1^\frac{1}{2} r_2^{-\frac{2}{3}} \frac{(\min\{|\mu_1|,|\mu_2|\} r_2)^\frac{1}{2} + (K r_2^2)^\frac{1}{3}}{(\max\{|\mu_1|,|\mu_2|\} r_1)^\frac{1}{2} + (K r_1^2)^\frac{1}{3}}. \label{spectrumpert2}
\end{equation}
\end{proposition}
\begin{proof}
If $\mu_1 = \mu_2$, both inequalities are trivial (since the operator norm is at most $1$).  In addition, since $E_x^{\mu_1}$ and $E_y^{\mu_2}$ are self-adjoint, it suffices to assume that $|\mu_1| > |\mu_2|$.  But for any $v$, $|E^{\mu_1}_x v| \leq |\mu_1|^{-1} |H_x v|$ (since $H_x$ is self-adjoint).  In addition, $|H_x E^{\mu_2}_y v| \leq |(H_y - H_x) E^{\mu_2}_y v| + |H_y E^{\mu_2}_y v|$.  But $||H_y - H_x|| \leq K |y-x|$ by the mean-value theorem (parametrizing the line segment from $x$ to $y$ as in \eqref{taylor}).  Combining all these inequalities gives \eqref{spectrumpert}.

Next, suppose $|x-y| \leq K^{-\frac{1}{3}} d^\frac{1}{3}$.  If $(|\mu_1| r_1)^\frac{1}{2} \leq (K r_1^2)^\frac{1}{3}$, then the inequality \eqref{spectrumpert2} is again trivial, so it may be assumed that this does not occur. Because the norms of the projections $E$ are one, the right-hand side of \eqref{spectrumpert} may be replaced by its square root.  Since $|x-y| \leq K^{-\frac{1}{3}} d^\frac{1}{3}$, so it must be the case that $|\mu_1|^{- \frac{1}{2}} (|\mu_2| + K |x-y|)^\frac{1}{2} \leq (|\mu_1|)^{-\frac{1}{2}} (|\mu_2|^\frac{1}{2} + K^\frac{1}{3} d^\frac{1}{6} )$.   But $|\mu_2|^\frac{1}{2} + K^\frac{1}{3} d^\frac{1}{6} \leq  (d r_2^{-1})^{\frac{1}{6}}(|\mu_2|^\frac{1}{2} + K^\frac{1}{3} r_2^\frac{1}{6})$ and $|\mu_1|^\frac{1}{2} \geq \frac{1}{2} (|\mu_1|^\frac{1}{2} + K^\frac{1}{3} r_1^\frac{1}{6})$.  Multiplying these estimates gives \eqref{spectrumpert2}.
\end{proof}

The following two norms will be the starting point for the construction of an appropriate family of nonisotropic balls adapted to the geometry of $\Phi$.  For any vector $v \in \R^n$ and any nonnegative $r>0$, let
\[ N_x [v,r] := r^{-1} \left( \sum_{\mu} |E_x^\mu v|^2 \left( (|\mu| r)^\frac{1}{2}+ (K r^2)^\frac{1}{3} \right)^2 \right)^\frac{1}{2} \]
and
\[ N_x^* [v,r] := \left(\sum_{\mu} \left(\frac{|E_x^\mu v|}{(|\mu| r)^\frac{1}{2}+ (K r^2)^\frac{1}{3}} \right)^2 \right)^\frac{1}{2}. \]
After a brief exposition of the elementary properties of these objects, the construction will be the following:  the distance from the point $x$ to the point $y$ will be measured by taking the infimum over $r > 0$ of all such $r$ for which $N_x[x-y,r] < 1$.  The dual object then measures the magnitude of (dual) vectors (i.e., the gradient of $\Phi$) in the appropriate nonisotropic sense; again the ``length'' of such an object $v$ being the infimum over all $r > 0$ for which $N_x^*[v,r]$.  But first, the basic properties of $N_x$ and $N_x^*$ which will be frequently exploited are proved:
\begin{proposition}
The following properties are true of $N_x$ and $N_x^*$: \label{normprop}
\begin{enumerate}
\item For fixed $x$ and $v$, $N_x[v,r]$ and $N_x^*[v,r]$ is a decreasing function of $r$.
\item For any $\theta \in (0,1]$,
\begin{equation}
\theta^{-\frac{1}{3}} N_x[v,r] \leq N_x[v,\theta r] \leq \theta^{-\frac{1}{2}} N_x[v,r],\label{scaling1}
\end{equation}
\begin{equation}
\theta^{-\frac{1}{3}} N^*_x[v,r] \leq N^*_x[v,\theta r] \leq \theta^{-\frac{1}{2}} N^*_x[v,r].\label{scaling2}
\end{equation}
\item Suppose that $x$ and $y$ are any two points in a Euclidean ball of radius $K^{-\frac{1}{3}} d^\frac{1}{3}$.  Then for any $r \leq d$,
\begin{equation} N_y[v, d^\frac{1}{3} r^\frac{2}{3}] \lesssim N_x[v,r], \label{normineq} \end{equation}
\begin{equation} N^*_y[v, d^\frac{1}{4} r^\frac{3}{4}] \lesssim N_x^*[v,r]. \label{dualineq} \end{equation}
\item Suppose $N_x[v] := \inf \set{r > 0}{ N_x[v,r] < 1}$ and likewise for $N^*_x[v]$.  Then 
\begin{equation}
 \left( N_x[v+w] \right)^\frac{1}{3} \leq \left( N_x[v] \right)^\frac{1}{3} + \left( N_x[w] \right)^\frac{1}{3}, \label{triangle}
\end{equation}
\begin{equation}
 \left( N^*_x[v+w] \right)^\frac{1}{2} \leq \left( N^*_x[v] \right)^\frac{1}{2} + \left( N_x^*[w] \right)^\frac{1}{2}. \label{dualtriangle}
\end{equation}
\item  For any two vectors $v,w \in \R^n$, 
\begin{equation} 
|v \cdot w| \leq r N_x[v,r] N_x^*[w,r] \label{csineq}
\end{equation}
moreover, for any $v$ there is a $w$ such that both sides are equal (and likewise with the roles reversed).  In addition,
\begin{equation}
|v \cdot H_x w| \leq r N_x[v,r] N_x[w,r]. \label{csineq2}
\end{equation}
\end{enumerate}
\end{proposition}
\begin{proof}
Properties 1 and 2 follow from an elementary inspection of the definition.  Property 3 is a consequence of \eqref{spectrumpert2}, via the triangle inequality.  For example,
\begin{align*}
 (N_y [v,r'])^2 & \leq n {r'}^{-2} \sum_\mu \sum_{\mu'} || E^\mu_y E^{\mu'}_x||^2 |E^{\mu'}_x v|^2 \left( (|\mu| r')^\frac{1}{2}+ (K {r'}^2)^\frac{1}{3} \right)^2  \\
& \leq 4 n^2 d^\frac{1}{3} r'^{-1} r^{-\frac{4}{3}} \sum_{\mu'} |E^{\mu'}_x v|^2 \left( (|\mu'| r)^\frac{1}{2}+ (K {r}^2)^\frac{1}{3} \right)^2\\
 & =  4 n^2 d^\frac{1}{3} r'^{-1} r^\frac{2}{3} (N_x[v,r])^2, 
\end{align*}
keeping in mind that \eqref{spectrumpert2} requires that $r'$ and $r$ be no greater than $d$. Taking $r' = d^\frac{1}{3} r^\frac{2}{3}$ gives \eqref{normineq}.  
As for \eqref{dualineq}, the reasoning is similar:
\begin{align*}
(N^*_x [v,r''])^2 & \leq n \sum_{\mu} \sum_{\mu'} \frac{||E_x^\mu E_y^{\mu'} ||^2 |E_y^{\mu'}v|^2}{\left((|\mu| r'')^\frac{1}{2}+ (K {r''}^2)^\frac{1}{3} \right)^2} \\
 & \leq 4 n^2 d^\frac{1}{3} {r} {r''}^{-\frac{4}{3}} \sum_{\mu'} \frac{|E_x^{\mu'}v |}{(|\mu'| r)^\frac{1}{2}+ (K r^2)^\frac{1}{3}} \\ 
& =  4 n^2 d^\frac{1}{3} {r} {r''}^{-\frac{4}{3}} (N_x^*[v,r])^2.
\end{align*}
This time, taking $r'' = d^\frac{1}{4} r^\frac{3}{4} \leq d$ gives \eqref{dualineq}.
To prove property 4, first observe that for any positive $\alpha, a, b,$ and $p > 1$, if $\frac{a}{b} \leq 1$, then \[\frac{a + \alpha a^p}{b + \alpha b^p} \leq \frac{a}{b}.\]
If $\phi_j(r) := (|\mu_j|^\frac{1}{2} r^{-\frac{1}{2}} + K^\frac{1}{3} r^{-\frac{1}{3}})^{-1}$ and $\tilde \phi_j(r) := (|\mu_j|^\frac{1}{2} r^{\frac{1}{2}} + K^\frac{1}{3} r^{\frac{2}{3}})$, it follows that
\[ \frac{\phi_j(r_1)}{\phi_j((r_1^\frac{1}{3}+r_2^\frac{1}{3})^3)} \leq \frac{r_1^\frac{1}{3}}{r_1^\frac{1}{3} + r_2^\frac{1}{3}} \mbox{ and } \frac{\tilde \phi_j(r_1)}{\tilde \phi_j((r_1^\frac{1}{2} + r_2^\frac{1}{2})^2)} \leq \frac{r_1^\frac{1}{2}}{r_1^\frac{1}{2} + r_2^\frac{1}{2}}. \]
Therefore, by convexity, the following inequality holds for any positive numbers $A_j,B_j$:
\[ \sum_j \left(\frac{A_j+B_j}{\phi_j((r_1^\frac{1}{3} + r_2^\frac{1}{3})^3)} \right)^2 \leq \frac{r_1^\frac{1}{3}}{r_1^\frac{1}{3}+r_2^\frac{1}{3}} \sum_{j} \left( \frac{A_j}{\phi_j(r_1)} \right)^2 + \frac{r_2^\frac{1}{3}}{r_1^\frac{1}{3}+r_2^\frac{1}{3}} \sum_{j} \left( \frac{A_j}{\phi_j(r_2)} \right)^2 \]
and likewise for $\tilde \phi_j$.  This gives the triangle inequalities \eqref{triangle} and \eqref{dualtriangle} when $A_j = |E_x^{\mu_j} v|$ and $B_j = |E_x^{\mu_j} w|$.
Finally, \eqref{csineq} and \eqref{csineq2} follow immediately from Cauchy-Schwartz.  Moreover, taking $w := \sum_\mu (|\mu|^\frac{1}{2} r^\frac{1}{2} + (K r^2)^\frac{1}{3})^2 E_x^\mu v$ gives equality.  
%
\end{proof}
%

Given the facts listed in proposition \ref{normprop}, the construction proceeds as follows: at each point $y \in \Omega$, there is a natural family of nonisotropic balls centered at $y$ which is induced by the norm $N_y$.  To be precise, let
\begin{equation}
B(y,r) := \set{x \in \R^n}{ N_y[x-y,r] < 1}
\label{balls}
\end{equation}
(note: for technical reasons and ease of proof, the balls $B(y,r)$ are taken to extend outside of $\Omega$ if $y$ is close to $\partial \Omega$ and/or $r$ is sufficiently large.)
One also makes the following definition for convenience:  given points $x,y \in \Omega$, let $d(x,y) := N_x[x-y]$ (as defined in property 4 of proposition \eqref{normprop}).
Proposition \ref{ballcomp} outlines some of the fundamental properties and relationships satisfied by this family of balls.  In short, the set $\Omega$ equipped with the balls $B(y,r)$ is a symmetric space in the sense of Coifman and Weiss \cite{cw1971} (just as in the work of Bruna, Nagel, and Wainger \cite{bnw1988}). 
The proofs are, for the most part, applications of the facts established in proposition \eqref{normprop}.
%

\begin{proposition}
The following facts are true of the family of balls \eqref{balls}:
\label{ballcomp}
\begin{enumerate}
\item For every $x \in \Omega$, the balls $B(x,r)$ are nested in the usual way:  $B(x,r) \subset B(x,r')$ when $r' > r$; moreover $B(x,r)$ is contained in the standard Euclidean ball of radius $K^{-\frac{1}{3}} r^\frac{1}{3}$.
\item The balls \eqref{balls} satisfy the doubling property, i.e., for any $B(x,r) \subset \Omega$, $|B(x,r)| \leq 2^\frac{n}{2} |B(x,\frac{1}{2}r)|$, where $|B(x,r)|$ denotes the Lebesgue measure of the ball $B(x,r)$.
\item For any $y \in B(x,r) \cap \Omega$, $B(x,r) \subset B(y, r')$ for some $r' \approx r$.
\item There is a covering of $B(x,r) \subset \Omega$ by balls $B(y,\delta r)$ such that the total number of balls in the covering depends only on $\delta$ and $n$.
\item For all $\delta \lless 1$, if $y \in \Omega$ is on the boundary of $B(x,r)$ (i.e., $r = d(x,y)$), then $B(y,\delta r)$ is contained in $B(x,2r)$ but does not intersect $B(x,\frac{1}{2} r)$.
\end{enumerate}
\end{proposition}
\begin{proof}
Both parts of property 1 are elementary:  the first follows from the fact that $N_x[v,r]$ is decreasing in $r$. 
The second part follows from the observation that $K^\frac{1}{3} r^{-\frac{1}{3}} |x-y| \leq N_x[y-x,r]$, which is less than $1$ if $y \in B(x,r)$.

Property 2 follows from \eqref{scaling1}; since $N_x[v,r]$ is sublinear in $v$, it must be the case that $N_x[2^{-\frac{1}{2}}(x-y), \frac{1}{2} r] \leq N_x[x-y,r]$ (and likewise for $N_x^*$). Taking the right-hand less than $1$ shows that $B(x,r)$ is contained in the standard Euclidean dilation of $B(x,\frac{1}{2} r)$ by a factor of $\sqrt{2}$, i.e., $B(x,r) \subset x + 2^\frac{1}{2}(B(x,\frac{1}{2}r)-x)$. 

As for property 3, $N_y[z-y,r] \leq N_y[x-y,r] + N_y[x-z,r] \lesssim N_x[x-y,r] + N_x[x-z,r]$.  Since the right-hand side is no greater than two, an application of \eqref{scaling1} gives $N_y[z-y,r'] < 1$ for some $r' \approx r$.  More generally, if $z \in B(x,\delta r)$ and $y \in B(x,r)$ for some $0 < \delta < 1$, then $N_y[z-x] \lesssim \delta^\frac{2}{3} r$ by \eqref{normineq}.
Thus the ball $B(x,\delta r)$, when translated to have center at $y$, is contained in the ball $B(y,\delta^\frac{2}{3} r')$ for some $r' \approx r$. 
%

To establish property four, notice that the ball $B(x,r)$ is an ellipsoid in $\R^n$; suppose it is given by \[B(x,r) = \set{y \in \R^n}{\sum_{j=1}^n r_j^{-2} (v_j \cdot (y-x))^2 < 1}\]
 for orthonormal vectors $v_j$ and radii $r_j$.  Let $\Lambda_\epsilon := x + \epsilon ( \Z r_1 v_1 + \cdots + \Z r_n v_n)$.  By \eqref{scaling1}, it must be the case that $x + \epsilon \sum_{j=1}^n \theta_j r_j v_j$ is in the ball $B(x,2n \epsilon^2 r)$ when $\sum_j \theta_j^2 < 2n$ and $\epsilon^2 \leq 2n$.  However, every point in $y \in B(x,r)$ is near to a point in $z \in B(x,r) \cap \Lambda_\epsilon$ (meaning that $y-z = \epsilon \sum_{j=1}^n \theta_j r_j v_j$ for some $\theta_j$'s of absolute value less than or equal to one). Therefore, the collection of all translates of $B(x,\epsilon^2 r)$ shifted to have centers at the points of $\Lambda_\epsilon \cap B(x,r)$ covers $B(x,r)$.  Thus, for some $r' \approx r$, the balls $B(y,(2 n \epsilon^2)^\frac{2}{3} r')$ for $y \in \Lambda_\epsilon \cap B(x,r)$ cover $B(x,r)$, and the number of such balls depends only on $\epsilon$ and $n$.

Finally, consider property 5.  
Combining \eqref{normineq} with \eqref{triangle} gives that, when $d(y,z) \leq d(x,y)$,
\[  \left| (d(x,z))^\frac{1}{3} - (d(x,y))^\frac{1}{3}  \right| \lesssim (d(x,y)^\frac{1}{3} d(y,z)^\frac{2}{3})^\frac{1}{3}, \]
(applying the  triangle inequality to $x-z = (x-z) + (z-y)$ or $x-y = (x-z) + (z-y)$ depending on whether the quantity in absolute values is positive or negative).  In particular, if $d(y,z)/d(x,y)$ is sufficiently small (in terms of $n$), then $d(x,z)/d(x,y)$ must be between $\frac{1}{2}$ and $2$.
\end{proof}

In later computations, it will be absolutely crucial to keep track of the amount to which a given ball $B(y,r)$ deviates from a standard Euclidean ball (of radius $K^{-\frac{1}{3}} r^\frac{1}{3}$).  The simplest way to record this information is to count how many of the eigenvalues of $H_y$ are big.  To that end, there are several new definitions in order.  First, 
given a ball $B(y,r)$, let $\locrank_s B(y,r)$ (called the $s$-local rank) be the dimension of the space spanned by all eigenvectors of $H_y$ with eigenvalues $\mu$ satisfying the inequality $|\mu| > s (K^2 r)^\frac{1}{3}$.
Given a ball $B(y,r)$, let $ \rank_s B(y,r)$ be the infimum of $\locrank_s B(x,r)$ for all $x \in B(y,r)$.

Another important consideration in what follows is the extent to which the spectrum of $H_y$ has large gaps.  Gaps will, in fact, be desirable, since then the big eigenvalues and the small ones will be easily distinguished (and accounted for separately).  To be more precise, $B(z,r)$ will be said to have a local spectral gap on $(a,b]$ if $\locrank_a B(z,r) = \locrank_b B(z,r)$; likewise $B(z,r)$ has a spectral gap on $(a,b]$ when $B(x,r)$ has a local spectral gap on $(a,b]$ for all $x \in B(z,r)$ (in which case $\rank_a B(z,r) = \rank_b B(z,r)$).  The next proposition formalizes the intuition that both high ranks and large spectral gaps must be preserved if one perturbs the center of the ball slightly:
\begin{proposition}
Suppose that $\locrank_s B(x,r) = k$.  Then for all $y \in B(x,\delta r) \cap \Omega$, $\locrank_s B(y,(1-s^{-1} \delta^\frac{1}{3})^3 r) \geq k$.  Moreover, $\rank_s B(x, (\frac{s}{s+1})^3 r) \geq k$ as well.
If, in addition, $B(x,r)$ has a local spectral gap on $(a,b]$.  Then for any $y \in B(x,\delta r) \cap \Omega$, $B(y,r)$ has a local spectral gap on $(a + \delta^\frac{1}{3}, b - \delta^\frac{1}{3}]$. \label{gap}
\end{proposition}
\begin{proof}
Let $U^s_x$ be projection onto the space at $x$ described in the definition of $s$-local rank, and let $L^{s}_x = I - U^s_x$.  As before, it must be the case that $||U^s_x v|| < s^{-1} (K^2 r)^{-\frac{1}{3}} ||H_x v||$ for any nonzero vector $v$.  Let $v$ be in the image of $L^{s'}_y$ for some $y$ and some $s$.  The mean-value theorem dictates that $|H_x v - H_y v| \leq K |x-y| |v|$, and $|H_y v| \leq s' (K^2 r)^{-\frac{1}{3}} |v|$ by virtue of the fact that $v = L^{s'}_y v$.  Thus, if $|x-y| \leq K^{-\frac{1}{3}} (\delta r)^\frac{1}{3}$, then $|U^s_x L^{s'}_y v| < s^{-1}(s'+\delta^{\frac{1}{3}}) |v|$ when the right-hand side is nonzero.  Since the operators are self-adjoint, it must also be the case that $|L^{s'}_y v| < s^{-1}(s'+\delta^{\frac{1}{3}}) |v|$ for any nonzero $v$ in the image of $U^s_x$. Fix $s' = s - \delta^\frac{1}{3}$.   By the triangle inequality, for any such $v$, $|U^{s'}_y v| > |v| - s^{-1}(s'+\delta^{\frac{1}{3}}) |v| = 0$, meaning that the total dimension at $y$ of eigenvectors with eigenvalues greater than $(s-\delta^\frac{1}{3}) (K^2 r)^\frac{1}{3}$ is at least $k$.  But this is equivalent to the statement that $\locrank_s B(y,(1-s^{-1} \delta^\frac{1}{3})^3 r) \geq k$.

The second statement follows from the observation that when $\delta = (1 + s^{-1})^{-3}$, then $\delta = (1 - s^{-1} \delta^\frac{1}{3})^3$, so the appropriate $s$-local rank condition holds for every point in $B(x,\delta r)$.

Finally, by a double application of the above reasoning, $\locrank_s B(x,r) \leq \locrank_{s - \delta^\frac{1}{3}} B(y,r) \leq \locrank_{s - 2 \delta^\frac{1}{3}} B(x,r)$.  Therefore, if $\locrank_a B(x,r) = \locrank_b B(x,r)$, then $\locrank_{a + \delta^\frac{1}{3}} B(y,r) = \locrank_{b - \delta^\frac{1}{3}} B(y,r)$.
\end{proof}

The final proposition of this section establishes that, in some sense, the property of being high-rank is complementary to the property of having a large spectral gap.  The informal idea is that, if the ball $B(x,r)$ does not have a large spectral gap, then by decreasing its radius by an appropriate factor, it becomes higher rank.  By an induction argument on rank, proposition \ref{decomp} will create a finite decomposition of any compact subset of $\Omega$ into regions where there is always some known ball with an arbitrarily large spectral gap (and, in fact, many such balls); in fact, the number of such regions will be completely independent of the particular choice of $H_x$.
\begin{proposition}
Fix any constants $0 < a < b < c$.  Fix any ball $B(x,r)$ whose closure is contained in $\Omega$ and any $z$ in the closure of that ball, and let $G_z \subset B(x,r)$ be the set of points $y$ such that either $B(z,d(z,y))$ or $B(y,d(z,y))$ has a local spectral gap on $(a,b]$.  If $b/c \lless 1$, then $B(x,r)\setminus G_z$ is covered by boundedly many balls (depending on $n,a,b,$ and $c$) of $c$-rank strictly greater than $\rank_{c} B(x,r)$. \label{decomp}
\end{proposition}
\begin{proof}  
For each integer $j$, let $I_j := [2^{-j} r, 2^{-j-1} r]$ (and neglect all negative $j$'s such that $B(x,r) \subset B(z,2^{-j}r)$).
Since the dimension is finite, there are only boundedly many $j$'s for which $I_j$ contains an $r'$ such that $B(z,r')$ fails to have a local spectral gap on $(a,b]$.  Let these exceptional scales be labeled $I_{j_1},\ldots,I_{j_N}$.

For a given $I_{j_k}$, cover $B(z,2^{-j_k} r) \setminus B(z,2^{-j_k-1} r)$ by boundedly many balls of radius $2^{-j_k} \delta r$ where $\delta$ is a fixed constant to be chosen suitably small.  Consider any such ball $B(w,\delta 2^{-j_k} r)$.  

Now suppose that $B(w,d(z,w))$ has a local spectral gap on $(a,b]$.  Then for all $y \in B(w,\delta 2^{-j_k} r)$, the ball $B(y,d(z,w))$ has a local spectral gap on $(a+\delta^\frac{1}{3},b-\delta^\frac{1}{3}]$.  Since $d(z,w)$ is within a factor of $2$ of $d(z,y)$, it must be the case that the ball $B(y,d(z,y))$ also has a spectral gap on $(2a,b/2]$ if $\delta$ is suitably small in terms of $a$ and $b$.

Suppose instead that $B(w,d(z,w))$ does not have a local spectral gap on $(a,b]$.  This means $\locrank_{a} B(w,d(z,w)) > \locrank_{b} B(w,d(z,w))$. Because $\delta$ was chosen suitably small, $\rank_a B(w,\delta 2^{-j_k} r) > \locrank_{b} B(w,d(z,w))$.  But for any $x \in B(w,\delta 2^{-j_k} r)$, it must be the case that $\locrank_{b'} B(x, \delta (a/b')^3 2^{-j_k} r) > \locrank_b B(w,d(z,w))$ for any $b' > b$.  Thus $B(w,\delta 2^{-j_k} r)$ may be covered by boundedly many balls (depending on $n$, $a$, and $b'$) of radius $\delta (a/b')^3 2^{-j_k} r$ on which the $b'$-rank is strictly greater than $\locrank_b B(w,d(z,w))$ (which is equal to $ \locrank_{b'} B(w,(b/b')^3 d(z,w)))$.  This quantity is at least equal to $\rank_{b'} B(x,r)$, if $b/b' \lless 1$ (to account for the fact that $d(z,w) \lesssim r$).
\end{proof}

\section{Integration-by-parts construction}
\label{ibpsec}
To begin the section, a few definitions are in order.  First, define the $C^k(B(y,r))$-norm of a function $f$ to be the supremum on $B(y,r)$ of $(v \cdot \nabla)^l f(x)$ where $l$ ranges from $0$ to $k$ and $v$ ranges over all vectors satisfying $N_y[v,r] \leq 1$.  
Next, suppose that the amplitude $\psi$ is compactly supported in $\Omega$.
Fix $0 < R_{max}$ to be smaller than the nonisotropic distance from the support of $\psi$ to $\partial \Omega$ (i.e.,  $B(x,R_{max}) \subset \Omega$ for all $x$ in the support of $\psi$; note that $R_{max}$ may be chosen to be any positive number less than $K$ times the third power of the Euclidean distance).

The purpose of this section is to establish the following result:  for any positive integer $N$, there exists a constant $C$ depending on $N$, $n$, and the $C^{N+1}$-norm of $\Phi$ on $\Omega$ such that
\begin{equation}  \left| \int e^{i \lambda \Phi(x)} \psi(x) dx \right| \leq C \left[\int_{\Omega} \frac{||\psi||_{C^N(B(y,N_y^*[\nabla \Phi(y)]))}  dy}{1 + (\lambda N_y^*[\nabla \Phi(y)])^N} + \frac{||\psi||_{C^N(\Omega)} |\Omega| }{(\lambda R_{max})^N} \right]. 
\label{sublevel}
\end{equation}
The main idea behind \eqref{sublevel} is, of course, an integration-by-parts procedure.  The goal will be to carry out the procedure on the largest possible region on which $\nabla \Phi$ is essentially constant.  In what follows, the balls $B(y,r)$ will serve as a suitable approximation to such a region; at the point $y$, the gradient of $\Phi$ is, for all intents and purposes, essentially constant on the ball $B(y,\rho(y))$.  To make these ideas precise, it is first necessary to establish a simple inequality analogous to Taylor's theorem to allow one to estimate how $\nabla \Phi(y)$ varies on the balls \eqref{balls}.  With that information in place, one can proceed to perform the integration-by-parts:
\begin{proposition}
Let $\mu$ be any real number.  Then
\begin{equation}
\left| E^\mu_y (\nabla \Phi(x)) - E^\mu_y (\nabla \Phi(y)) \right| \leq |\mu| |E^\mu_y (x-y)| + \frac{1}{2} K |x-y|^2. \label{hessianspect}
\end{equation}
Moreover, if $v_j$ is any unit eigenvector of $H_y$ with eigenvalue $\mu$, then
\begin{equation}
 \left| v \cdot \nabla \Phi(x) - v \cdot \nabla \Phi(y) - \mu v \cdot (x-y) \right| \leq \frac{1}{2} K | x - y|^2. \label{hessian}
\end{equation}
\end{proposition}
\begin{proof}
Begin with the following formula:  given any twice-differentiable function $f$ defined on $[0,1]$,
\begin{equation}
f(1) = f(0) + f'(0) + \int_0^1 (1-t) f''(t) dt. \label{taylor}
\end{equation}
This formula is just the fundamental theorem of calculus after an integration-by-parts.  Now, fix any unit vector $v$ and any points $x,y \in \Omega$.  Since $\Omega$ is convex, the function $f(t) = v \cdot \nabla \Phi( t x + (1-t) y)$ is defined on $[0,1]$ and twice-differentiable when $\Phi$ has continuous derivatives through the third order.  Differentiation gives that $f'(0) = v \cdot H_y (x-y) $, where $H_y$ is the Hessian of $\Phi$ at $y$; moreover, $|f''(t)| \leq K |x-y|^2$, where $K$ is the constant described in the proposition.  Therefore, by \eqref{taylor},
\[ |v \cdot \nabla \Phi(x) - v \cdot \nabla \Phi(y) - v \cdot H_y (x-y)| \leq \frac{1}{2} K |x-y|^2. \]
Now \eqref{hessian} follows trivially from this inequality since in this case $v \cdot H_y (x-y) = \mu v \cdot (x-y)$.  Moreover, if one instead takes $v = E^\mu_y (\nabla \Phi(x) - \nabla \Phi(y))$ (appropriately normalized), then \eqref{hessianspect} follows by observing that $|v \cdot H_y (x-y)| \leq |\mu||E^\mu_y (x-y)|$ and applying the triangle inequality. 
\end{proof}

\begin{proposition}
When $d \lless N_y^*[\nabla \Phi(y)]$, $N_x^*[\nabla \Phi(x)] \approx N_y^*[\nabla \Phi(y)]$ for all $x \in B(y,d)$. \label{comp}
\end{proposition}
\begin{proof}
By the previous proposition, given $x$ and $y$ in $\Omega$ within Euclidean distance $K^{-\frac{1}{3}} d^\frac{1}{3}$, 
$|\nabla \Phi(x) - \nabla \Phi(y) - H_y (x-y)| \leq \frac{1}{2} (K d^2)^\frac{1}{3}$.  By definition of $N_y^*$, it follows that $N_y^*[\nabla \Phi(x) - \nabla \Phi(y) - H_y (x-y)] \lesssim d$.  Moreover, by \eqref{csineq}, the quantity $N_y^*[H_y(x-y)]$ must equal $\inf \set{r > 0}{|w \cdot H_y(x-y)| < r \ \forall w \mbox{ s.t. } N_y[w] \leq r}$.  Combined with \eqref{csineq2}, it follows that $N_y^*[H_y (x-y)] \lesssim d$ as well.  Therefore the triangle inequality \eqref{dualtriangle} implies that
\[ \left|(N_y^*[\nabla \Phi(y)])^\frac{1}{2} - (N_y^*[\nabla \Phi(x)])^\frac{1}{2} \right| 
\lesssim d^\frac{1}{2}.\]
Therefore if $d \lless N_y^*[\nabla \Phi(y)]$ (let $r := N_y^*[\nabla \Phi(y)]$), $N_y^*[\nabla \Phi(x),r] \approx 1$.  But then applying \eqref{dualineq} in both directions implies that $N_y^*[\nabla \Phi(x),r] \approx N_x^*[\nabla \Phi(x),r]$, which means that $N_y^*[\nabla \Phi(y)] \approx N_x^*[\nabla \Phi(x)]$.
\end{proof}

\begin{proposition}
Suppose $\psi$ is a $C^\infty$ amplitude supported on $B(y,d)$ for some $d \lless N_y^*[\nabla \Phi(y)]$ and $d \leq R_{max}$.  \label{ibplemma} Then for any positive integer $N$, there is a constant $C$ depending on $N$, $n$, $R_{max}$, and the $C^{N+1}$-norm of $\Phi$ on $\Omega$,  for which
\[ \left| \int e^{i \lambda \Phi(x)} \psi(x) dx \right| \leq C  ||\psi||_{C^N(B(y,d))} |B(y,d)| (\lambda d)^{-N}. \]
\end{proposition}
\begin{proof}
By the reasoning of the previous proposition, $N_y^*[\nabla \Phi(x)-\nabla \Phi(y)] \lesssim d$ when $x \in B(y,d)$, which means that $|v \cdot (\nabla \Phi(x) - \nabla \Phi(y))| \lesssim d$ for any vector $v$ satisfying $N_y[v,d] \leq 1$ (by \eqref{scaling2} and \eqref{csineq}). Moreover, by \eqref{csineq}, there exists a vector $v$ satisfying $N_y[v,d] \leq 1$ such that $v \cdot \nabla \Phi(y) = d N_y^*[\nabla \Phi(y),d]$.  Thus \eqref{scaling2} implies that when $d \lless N_y^*[\nabla \Phi(y)]$, $v \cdot \nabla \Phi(x) \gtrsim d$ for all $x \in B(y,d)$.
Now consider the following differential operator:
\[L(f)(x) := -i \frac{v \cdot \nabla f(x)}{v \cdot \nabla \Phi(x)}.\]
The denominator is never zero, so $L$ is well-defined and smooth throughout $B(y,d)$.
Furthermore, $L(e^{i \lambda \Phi}) = \lambda \Phi$ on $B(y,d)$.   If one takes $L^t$ to be the transpose of $L$, integration-by-parts guarantees that, for all nonnegative integers $N$, 
\begin{equation}
\int e^{i \lambda \Phi(x)} \psi(x) dx = \lambda^{-N} \int e^{i \lambda \Phi(x)} \left(L^t\right)^N \psi(x) dx. \label{transpose}
\end{equation}
The key estimate needed to understand $(L^t)^N$ is an estimate of the size of $(v \cdot \nabla)^k \Phi(x)$ on $B(y,d)$ when $k \geq 2$.  In particular, one would like to show that $(v \cdot \nabla)^k \Phi(x)$ is of the same magnitude as $d^k$.  Since $N_y[v] \leq d$, the inequalities \eqref{csineq2} and \eqref{dualineq} give that $|(v \cdot \nabla)^2 \Phi(x)| \lesssim d$ on $B(y,d)$ (since the second derivative is precisely $v \cdot H_x v$).  When $k \geq 3$, $|(v \cdot \nabla)^k \Phi(x)| \leq C_k d^\frac{k}{3}$ where $C_k$ depends on the $C^k$-norm of $\Phi$ on $\Omega$; this simply follows from the observation that $|v| \leq K^{-\frac{1}{3}} d^\frac{1}{3}$.  Since $d \leq R_{max}$, $d^\frac{k}{3} \leq (R_{max})^{\frac{k-3}{3}} d$ for all $k \geq 3$.

Now, using the estimates for $|(v \cdot \nabla) \Phi(x)|$ and the Leibniz rule, it is easily established that 
$|(L^t)^N \psi(x,\xi)|\lesssim C d^{-N} ||\psi||_{C^N(B(y,d))}$, where $C$, as anticipated, depends on $N$, $n$, the $C^{N+1}$-norm of $\Phi$ on $\Omega$, and $R_{max}$.
%
%
%
%
Taking absolute values on the right-hand side of \eqref{transpose} and making an $L^\infty$ estimate on the ball $B(y,d)$ gives the desired conclusion. 
\end{proof}

To apply proposition \ref{ibplemma} to the general situation \eqref{oscillate}, it is necessary to create a partition of unity.  Rather than attempting to decompose the support of $\psi$ into countably many, essentially disjoint balls $B(y,r)$, a simpler approach is to make the partition continuous, adapted to the balls $B(y,c_n N_y^*[\nabla \Phi(y)])$ for {\it each} $y$ in the support of $\psi$.  
For each $y \in \Omega$, let $r(y) := \min\{N_y^*[\nabla \Phi(y)],R_{max}\}$.  By proposition \ref{comp}, after $r$ is multiplied by a suitably small constant depending only on $n$, $r(x) \approx r(y)$ whenever $x$ and $y$ are $B(x,r(x)) \cap B(y,r(y))$ has nonempty intersection in $\Omega$.

Fix some smooth $\phi$ supported on $[-1,1]$ which is identically one on $[-\frac{1}{2},\frac{1}{2}]$.  For each $y$ in the support of $\psi$, let $\eta_y$ be a smooth cutoff function on $\R^n$ given by
$\eta_y(x) := |B(y,r(y))|^{-1} \phi(N_y[x-y,r(y)])$.
This cutoff function is necessarily supported on $B(y,r(y))$.  Moreover, $\eta_y$ is identically equal to $|B(y,r(y))|^{-1}$ on $B(y,\frac{r(y)}{4})$ by \eqref{scaling2}.

As for smoothness, for fixed $y$, let $\nabla_x$ be the gradient in the $x$ variable.  By the definition of $N_y$,
\[ v \cdot \nabla_x (N_y[x-y,r(y)])^2 = \sum_\mu (E_y^\mu v) \cdot E_y^\mu (x-y) \left( \left(\frac{|\mu|}{r(y)}\right)^\frac{1}{2} + \left(\frac{K}{r(y)}\right)^\frac{1}{3} \right)^2, \] 
which is less than or equal to $N_y[v,r(y)] N_y[x-y,r(y)]$ by Cauchy-Schwartz.  Applying the chain rule,  the $C^k(B(y,r(y))$-norm of $\eta_y$ must uniformly bounded by $C_{k,n} |B(y,r(y))|^{-1}$ where $C_{k,n}$ depends only on $k$ and $n$.  Moreover, the $C^k(B(x,r(x))$-norm of $\eta_y$ must be similarly uniformly bounded whenever $x \in B(y,r(y))$.
%
%
Finally, note that  propositions \ref{ballcomp} and \ref{comp} give that $|B(x,r(x))| \approx |B(y,\rho(y))|$ for all $x$ in the support of $\eta_y$.  This is because every ball of radius comparable to $r(y)$ and centered at $y$ is contained in a ball of radius comparable to $r(y)$ centered at $x$ and vice-versa.

Now consider the new function $\Psi$ given by
\[ \Psi(x) := \int_\Omega \eta_{y}(x) dy. \]
For a fixed $x \in \Omega$, the support of the integral for $\Psi$ is contained in $B(x,r')$ for some $r' \approx r(x)$; this is because every $y$ with $\eta_y(x) \neq 0$ must have the property that $x \in B(y,r(y))$, so that $y \in B(x,r')$ for some $r' \approx r(y) \approx r(x)$.
Likewise, the integrand is identically one on some ball $B(x,r'')$ with $r'' \approx r(x)$.  This is because $\eta_y (x) = 1$ whenever $x \in B(y,\frac{1}{4}\rho(y))$, which is always true when $y \in B(x, r'')$ for some $r'' \lless r(x)$.  Now, given that $|B(x,r(x))| \approx |B(y,r(y))|$ in the support of the integral as well, it follows that $\Psi(x) \approx 1$ because the supremum of $\eta_y(x)$ times the size of the support is bounded by some constant depending only on $n$, but also there is a set, namely $B(x,r'')$ from above, which is completely contained in $\Omega$ which has size also comparable to $|B(x,r(x))|$ and the integrand is comparable to $|B(x,r(x)|^{-1}$ there.  (Note that for $B(x,r'') \subset \Omega$ it is necessary to further restrict $r(x) \lless R_{max}$.)


It follows that the reciprocal of $\Psi$ is well-defined.  Moreover, for any positive integer $k$, the $C^k(B(y,r(y)))$-norms of $\Psi$ and $\Psi^{-1}$ must be uniformly bounded by some constant $C$ depending only on $k$ and $n$.

To complete the proof of \eqref{sublevel}, Fubini's theorem dictates that
\[ \int e^{i \lambda \Phi(x)} \psi(x) dx = \int_\Omega \! \int e^{i \lambda \Phi(x)} \frac{\psi(x) \eta_{y}(x)}{\Psi(x)} dx dy, \]
Estimating the inner integral by proposition \ref{ibplemma} (using the case $N=0$ when $\lambda r(y) \leq 1$ and the general case elsewhere) and using the derivative estimates for $\eta_y$ and $\Psi^{-1}$ gives the inequality
\begin{equation} \left| \int e^{i \lambda \Phi(x)} \psi(x) dx \right| \leq C_N \int_{\Omega}  \frac{||\psi||_{C^N(B(y,r(y)))} dy}{1 + (\lambda r(y))^N}, \label{sublevel2} \end{equation}
where $C_N$ depends on $N$, $n$, the $C^{N+1}$-norm of $\Psi$ on $\Omega$, and $R_{max}$. 
Since the $C^k(B(y,r))$-norm of $\Psi$ increases as $r$ increases, \eqref{sublevel} holds.

\section{Proof of theorem 1}
\label{gradientsec}
In light of the results of the previous section, especially the inequality \eqref{sublevel}, to prove theorem \ref{oscthm} (which is completely local) it suffices to assume that $\psi$ is supported on some ball $B(y,d)$, and that $\Omega = B(y,(1+\epsilon)d)$ for any small $\epsilon > 0$.
Recall the nondegeneracy condition \eqref{nondegen}:  for each $x \in \Omega$ and each finite $\mu$, let $V_{\mu,x}$ be the vector space of eigenvectors of $H_x$ with eigenvalues $\nu$ satisfying $|\nu| \leq \mu$.  From here forward, it will be assumed that there is a constant $K' > 0$ such that, for any $\mu$ (with magnitude less than some prescribed maximum $M$) and any unit vector $v \in V_{\mu,x}$, there is a unit vector $w \in V_{R \mu,x}$ such that
\begin{equation*} (v \cdot \nabla) (v \cdot \nabla) (w \cdot \nabla) \Phi(y) \geq K'
\end{equation*}
for all $y \in \Omega$.   The goal of this section is to show that, when the nondegeneracy condition holds, there are a fixed, bounded number of points $z_i \in \Omega$ such that the nonisotropic magnitude of $\nabla \Phi(y)$ at $y$ scales linearly in the nonisotropic distance from $y$ to one of the points $z_i$.  Once this fact is established, a standard dyadic decomposition of $\Omega$ into balls $B(z_i,2^{-j})$ will be used to estimate the right-hand side of \eqref{sublevel} and yield theorem \ref{oscthm}.

To accomplish this goal, some of the estimates already established (for example, \eqref{dualineq}) will need to be refined slightly to reflect the distinct behaviors encountered when differentiating $\Phi$ in directions corresponding to large-eigenvalue eigenvectors as compared to those directions with small eigenvalues.  It is in this section that the existence of spectral gaps will be exploited.  To begin, a stronger form of \eqref{dualineq} is established allowing for even more favorable comparisons between $N_x^*[w]$ and $N_z^*[w]$ when $w$ sits in the span of the large-eigenvalue eigenvectors:
\begin{proposition}
Fix $x$ and $z$ in $\Omega$ with $|x-z| \leq K^{-\frac{1}{3}} d^\frac{1}{3}$, and fix any constant $\beta > 0$.  Suppose $w \in \R^n$ satisfies $E_z^\mu w = 0$ for all $|\mu| \leq \beta (K^2 d)^\frac{1}{3}$.  Then for all $r \leq d$,
\begin{equation}
N_z^*[w,r] \lesssim (1 + \beta^{-1}) N_x^*[w,r]. \label{dualineqplus}
\end{equation}
\end{proposition}
\begin{proof}
The proof is a minor modification of the proof of \eqref{dualineq}.  The major difference is that \eqref{spectrumpert2} is replaced by \eqref{spectrumpert} and  $(\max\{|\mu|,|\mu'|\})^{-1}$ is replaced by $(1 + \beta)(\beta |\mu'| + |\mu|)^{-1}$. More precisely:
\begin{align*}
(N_z^*[w,r])^2 & \leq \sum_{|\mu| > \beta (K^2 d)^\frac{1}{3}} \frac{|E_z^\mu w|^2}{|\mu| r} \\
& \leq n (1+\beta) \sum_{|\mu| > \beta (K^2 d)^\frac{1}{3}} \sum_{\mu'} \frac{|\mu| + (K^2 r)^\frac{1}{3}}{\beta |\mu'| + |\mu|} \frac{|E_x^{\mu'} w|^2}{|\mu| r} \\
& \leq n \frac{(1+\beta)^2}{\beta} \sum_{|\mu| > \beta (K^2 d)^\frac{1}{3}} \sum_{\mu'} \frac{|E_x^{\mu'} w|^2}{|\mu| r + \beta |\mu'| r} \\
& \leq n^2 \frac{(1+\beta)^2}{\beta^2} \sum_{\mu'} \frac{|E_x^{\mu'} w|^2}{(K r^2)^\frac{2}{3} + |\mu'| r} \leq 2 n^2 \frac{(1+\beta)^2}{\beta^2} (N_x^*[w,r])^2
\end{align*}
(the final line is true because $|\mu| r > \beta (K^2 r)^\frac{1}{3} \geq \beta (K^2 d)^\frac{1}{3}$).
\end{proof}
To simplify notation somewhat, let $E^+_y$ be projection onto the space spanned by the eigenvectors of $H_y$ which have eigenvalues in magnitude greater than $\beta (K^2 d)^\frac{1}{3}$, where $\beta$ is (for the moment) any fixed, positive real number.  The following proposition accomplishes the desired estimate of this section (namely, that the nonisotropic length of $\nabla \Phi$ scales like nonisotropic distance) when the points under consideration have a displacement vector which points in essentially the ``large-eigenvalue'' directions.  In particular, this situation is sufficiently favorable that there is no need to appeal to the nondegeneracy condition here:
\begin{proposition}
Suppose $x, z \in B(y,d) \cap \Omega$.  If $\beta \ggreater 1$ and $E_y^+ (x-z) = x-z$, then
\begin{equation}
N_x^*[\nabla \Phi(x)] + N_z^*[\nabla \Phi(z)] \gtrsim N_y[x-z]. \label{hesssize1}
\end{equation}
Moreover, fix any $\delta > 0$.  When $\beta \ggreater 1+ \delta^{-1}$, then for every $x$ on the boundary of $B(y,d)$ which satisfies $N_y[E_y^+(x-y),d] \geq \delta$, 
\begin{equation}
N_x^*[\nabla \Phi(x)] + N_y^*[\nabla \Phi(y)] \gtrsim d. \label{hesssize2}
\end{equation}
\end{proposition}
\begin{proof}
For contradiction, assume $N_x^*[\nabla \Phi(x)] + N_z^*[\nabla \Phi(z)] \lless N_y[x-z]$.
In this case, it suffices to prove that $N_x^*[\nabla \Phi(x) - \nabla \Phi(z)] \gtrsim N_y[x-z]$.  This is because
\[(N_x^*[\nabla \Phi(x) - \nabla \Phi(z)])^\frac{1}{2} \leq (N_x^*[\nabla \Phi(x)])^\frac{1}{2} + (N_x^*[\nabla \Phi(z)])^\frac{1}{2} \]
by \eqref{dualtriangle} and $N_x^*[\nabla \Phi(z)] \lesssim (N_y[x-z])^\frac{1}{3} (N_z^*[\nabla \Phi(z)])^\frac{3}{4}$ by \eqref{dualineq} (since $x$ and $z$ must be contained in a Euclidean ball of radius $K^{-\frac{1}{3}} (N_y[x-z])^\frac{1}{3}$).  With the assumption $N_x^*[\nabla \Phi(x) - \nabla \Phi(z)] \gtrsim N_y[x-z]$, it must be the case that $N_x^*[\nabla \Phi(x)] \gtrsim N_y[x-z]$ when $N_z^*[\nabla \Phi(z)] \lless N_y[x-z]$.  The argument is similar for \eqref{hesssize2}, making it necessary to show that $N_x^*[\nabla \Phi(x) - \nabla \Phi(y)] \gtrsim d$, which is the same inequality needed for \eqref{hesssize1} if $z=y$.

By \eqref{dualineqplus}, it suffices to show that $N_y^*[E_y^+(\nabla \Phi(x) - \nabla \Phi(z))] \gtrsim N_y[x-z]$.
To establish this inequality, the following variant of \eqref{hessianspect} is needed:  when $x,z \in B(y,d)$, then
\[ |\nabla \Phi(x) - \nabla \Phi(z) - H_y (x-z)| \leq (K^2 d)^\frac{1}{3} |x-z|. \]
The proof is essentially the same as the proof of \eqref{hessianspect}, but is based on a slightly different application of the fundamental theorem of calculus, namely
\[ \nabla \Phi(x) - \nabla \Phi(z) = H_y (x-z) + \int_0^1 \int_0^1 \frac{d}{d \varphi} H_{\varphi ( \theta x + (1-\theta)z) + (1-\varphi)y} (x-z) d \varphi d \theta. \]
It follows from now standard arguments that
\begin{align*}
 (N_y^*[E_y^+ (\nabla \Phi(x) - \nabla \Phi(z) - H_y (x-z)),r])^2 & \leq \frac{n}{\beta} (K^2 d)^\frac{1}{3} r^{-1} |x-z|^2
\end{align*}
for any $r \leq d$.  But $|x-z|^2 = \sum_\mu |E_y^\mu(x-z)|^2$; breaking the sum into big and small $\mu$, it follows that
\begin{align} 
(N_y^*[E_y^+ (\nabla \Phi(x) - \nabla \Phi(z) - H_y (x-z)),r])^2 & \nonumber \\ \leq \frac{n}{\beta^2} (N_y[E_y^+(x-z),r])^2 + \frac{n}{\beta} &  d^\frac{1}{3} r^{-\frac{1}{3}} (N_y[E_y^{-} (x-z),r])^2. \label{rhs00}
\end{align}
Since $|\mu| > \beta (K^2 d)^\frac{1}{3}$, elementary manipulations give
\[ \frac{|\mu|}{\left((|\mu| r)^\frac{1}{2} + (K r^2)^\frac{1}{3} \right)} \geq \frac{\beta}{2(1+\beta)} r^{-1} \left((|\mu| r)^\frac{1}{2} + (K r^2)^\frac{1}{3} \right) \]
for any $r \leq d$.  Applying this inequality to the norms $N_y^*$ and $N_y$, it must be the case that
\begin{equation} N_y^*[E_y^+ H_y (x-z), r] \geq \frac{\beta}{2(1+\beta)} N_y[E_y^+ (x-z),r]. \label{rhs01}
\end{equation}
%
To prove \eqref{hesssize1}, simply observe in \eqref{rhs00} that $E_y^+(x-z) = (x-z)$ and $E_y^-(x-z) = 0$.  Combining \eqref{rhs00} and \eqref{rhs01}, and using the triangle inequality gives \eqref{hesssize1} by taking $r = N_y[x-z]$ and applying \eqref{dualineqplus}.  As for \eqref{hesssize2},  fixing $r = d$ and $z = y$, the right-hand side of \eqref{rhs01} is bounded away from zero when $\beta \ggreater \delta^{-1}$, and the right-hand side \eqref{rhs00} is bounded away from one when $\beta \ggreater 1$.
\end{proof}

Now the second half of the goal at hand must be accomplished; namely, the nonisotropic norms of $\nabla \Phi(x)$ and $\nabla \Phi(y)$ must be compared when $x-y$ does not point in a ``large-eigenvalue'' direction.  It is at this point that the nondegeneracy condition applies.  From the previous proposition, we may assume that $N_y[E_y^+(x-y),d] \leq \delta$, where $d$ is the distance from $y$ to $x$ and $\delta$ is a nonnegative parameter to be specified.  Suppose $B(y,d)$ has a spectral gap on $(1,\beta]$.   If $d$ is sufficiently small, then the nondegeneracy condition implies the existence of a vector $w \in V_{\mu}$ with $\mu = R (K^2 d)^\frac{1}{3}$ such that $(w \cdot \nabla) (E_y^- (x-y) \cdot \nabla)^2 \Phi(z) \geq K'|E_y^-(x-y)|^2$ for all $z \in \Omega$, where $E_y^- (x-y) := x-y - E_y^+(x-y)$.  This $w$ necessarily satisfies $N_y[w,d] \leq (1+R^\frac{1}{2}) |w| (K d^{-1})^\frac{1}{3}$ by virtue of the fact that $w$ is a unit vector lying in $V_\mu$.  Moreover, $|E_y^+(x-y)| \leq \delta \beta^{-\frac{1}{2}} (K^{-1} d)^\frac{1}{3}$ while  $|E_y^-(x-y)| \geq (\frac{1-\delta^2}{2})^\frac{1}{2} (K^{-1} d)^\frac{1}{3}$.  Thus, when $\beta \ggreater \delta^{-2}$, it must also be the case that $(w \cdot \nabla)((x-y) \cdot \nabla)^2 \Phi(z) \geq \frac{1}{2} K' |w| |x-y|^2$.  The bottom line of these calculations is that the nondegeneracy condition on $\Phi$ guarantees that the hypotheses of the following proposition hold; as a result the nonisotropic norm can again be favorably estimated:
\begin{proposition}
Fix any ball $B(y,d) \subset \Omega$, and any $x$ on the boundary of $B(y,d)$.  Suppose there exists a vector $w$ such that $(w \cdot \nabla)((x-y) \cdot \nabla)^2 \Phi(z) \geq K' |w| |x-y|^2$ for all $z \in B(y,d)$ and $N_y[w,d] \leq \gamma |w| (K d^{-1})^\frac{1}{3}$.  Then for some $\alpha \lless K'/K$ and $\delta \lless K'/(K \gamma)$, if $B(y,d)$ has a spectral gap on $(\alpha,\beta]$ and $N_y[E^+_y(x-y),d] \leq \delta$, then
\[ N_x^*[\nabla \Phi(x)] + N_y^*[\nabla \Phi(y)] \gtrsim \left(\frac{K'}{K \gamma} \right)^3 d. \]
\end{proposition}
\begin{proof}
When $B(y,d)$ has a local spectral gap on $(\alpha,\beta]$, then as noted above
\[ 1 - \delta^2 \leq (N_y[E_y^- (x-y),d])^2 \leq  (1 + \alpha^\frac{1}{2})^2 (K d^{-1})^\frac{2}{3} |E_y^- (x-y)|^2. \]
Since $|x-y| \leq K^{-\frac{1}{3}} d^\frac{1}{3}$, it must therefore hold that
\[ |x-y| \leq \frac{1 + \alpha^\frac{1}{2}}{(1 - \delta^2)^\frac{1}{2}} |E_y^- (x-y)|. \]
The mean-value theorem can be applied to estimate $|w \cdot(\nabla \Phi(x) - \nabla \Phi(y)) - w \cdot (H_y (x-y))|$ in terms of the pointwise values of $(w \cdot \nabla)((x-y) \cdot \nabla)^2 \Phi(z)$.  Provided $\alpha, \delta \lless 1$, the term $|x-y|^2$ on the right-hand side of this comparison may be replaced by $|E_y^-(x-y)|^2$, giving
\[ |w \cdot(\nabla \Phi(x) - \nabla \Phi(y)) - w \cdot (H_y (x-y))| \gtrsim  K' |w| |E_y^- (x-y)|^2. \]
By \eqref{csineq2} and the given spectral gap on $B(y,d)$, one may estimate the Hessian term in two pieces:
\[ |w \cdot (H_y E_y^-(x-y))| \leq \alpha |w| (K d^2)^\frac{1}{3} \lesssim K \alpha |w| |E_y^-(x-y)|^2 \]
and
\[ |w \cdot (H_y E_y^+(x-y))| \leq \delta d N_y[w,d]. \]
Thus, when $\alpha \lless K'/K$, the inequality \eqref{csineq2} gives that
\[ \frac{K' |w| (K^{-1} d)^\frac{2}{3}}{d N_y[w,d]} \lesssim N_y^*[\nabla \Phi(x) - \nabla \Phi(y),d] + \delta. \]
Fixing $\delta \lless K'/(K \gamma)$ gives the conclusion of the proposition.
\end{proof}

In the context of the proof at hand, the main consequence of the previous two propositions is as follows.  Fix any two points $x,z \in \Omega$ with $d(z,x) = r$; if the nondegeneracy condition \eqref{nondegen} holds and $B(z,r)$ or $B(x,r)$ has a sufficiently large spectral gap (depending on the dimension and on the constants in the nondegeneracy condition), then 
\[ N_x^*[\nabla \Phi(x)] + N_z^*[\nabla \Phi(z)] \gtrsim r. \]
With this fact in hand, one appeals to proposition \ref{decomp} inductively as follows:  Suppose that $B(y,d)$ for some $d$ sufficiently small and its closure is contained in $\Omega$.  Let $z$ be any point in the closure at which $N_x^*[\nabla \Phi(x)]$ attains the minimum.  By proposition \ref{decomp}, the subset of $x \in B(y,d)$ on  for which $N_x^*[\nabla \Phi(x)] + N_z^*[\nabla \Phi(z)]$ is not greater than a constant times $d(z,x)$ may be covered by boundedly many balls whose $\beta'$-rank is strictly greater than $\rank_{\beta'} B(y,d)$ (here $\beta' \ggreater \beta$ as specified by proposition \ref{decomp}).  By induction on rank, the following is true:  there exist a bounded number of points (depending on $n$, $K$, $K'$, and $M$ from the nondegeneracy condition) $z_1,\ldots,z_l$ and balls $B(y_l,d_l)$ such that $z_i$ is in the closure of $B(y_i,d_i)$, the balls $B(y_i,d_i)$ cover $B(y,d)$, and for all $x \in B(y_i,r_i)$, $N_x^*[\nabla \Phi(x)] \geq N_{z_i}^*[\nabla \Phi(z_i)]$ and $N_x^*[\nabla \Phi(x)] + N_{z_i}^*[\nabla \Phi(z_i)] \gtrsim d(z_i,x)$.   Consequently $N_x^*[\nabla \Phi(x)] \gtrsim d(z_i,x)$ as well.
Thus, it must be the case that there exists $C$ depending on the constants in \eqref{nondegen}, on $n$, and on $K$, and boundedly many points $z_i$ such that
\[ \int_{B(y,d)} \frac{dx}{1 + (\lambda N_x^*[\nabla \Phi(x)])^N } \leq C \sum_{i} \int_{B(y_i,d_i)} \frac{dx}{1 + (\lambda d(z_i,x))^N}. \]

In the usual dyadic decomposition of the range of $d(z_i,\cdot)$, it follows that the right-hand side is bounded above by
\begin{equation} C \sum_{i} \sum_{j=0}^{j_i} 2^{-Nj} |B(z_i,2^j \lambda^{-1})|. \label{geom}
\end{equation}
where the sum over $j$ is truncated at index the index $j_i$ such that $B(y_i,d_i) \subset B(z_i, 2^{j_i-1} \lambda^{-1})$.  Let $k$ be the dimension of $V_{M,y}$.  Since the ball $B(z_i,2^{j} \lambda^{-1})$ is an ellipsoid, its volume can be computed explicitly; in particular,
\[ |B(z_i,2^j \lambda^{-1})| \lesssim (M^{-1} 2^j \lambda^{-1})^\frac{n-k}{2} (K^{-1} 2^j \lambda^{-1})^\frac{k}{3}. \] 
Choosing $N$ larger than $\frac{n-k}{2} + \frac{k}{3}$ in \eqref{geom} makes the sum convergent, and gives precisely the estimate claimed by theorem \ref{oscthm}.

\section{Proof of theorem 2}
\label{genericsec}
The final topic to be addressed concerns the behavior of a generic function vanishing to third order at some point.  Let ${\mathfrak S}_n^3$ be the vector space of homogeneous cubic polynomials with real coefficients in $n$ variables.  A smooth phase $\Phi$ will be called generic when there exists $p \in {\mathfrak S}_n^3$ in some generic subset (in the standard meaning of generic) such that $\Phi - p$ vanishes to fourth order at the origin.

As in the work of Greenleaf, Pramanik, and Tang \cite{gpt2007}, an interesting simplification of an algebraic nature occurs when estimating \eqref{oscillate} in the presence of a large number of dimensions.  The simplification arises because the Hessian matrix $H_y$ of a generic cubic polynomial (or a homogeneous polynomial of any degree, for that matter) will always have high rank unless $y=0$.  This situation is ideal for making uniform estimates of \eqref{oscillate} because $\Phi(x)$ and $\Phi(x) + \xi \cdot x$ share the same Hessian matrices, and therefore give rise to the same nonisotropic family of balls.

Suppose $\locrank_\beta B(y,d) = k$ and $D$ is the product of the absolute value of those eigenvalues $\mu$ of $H_y$ satisfying $|\mu| > \beta (K^2 d)^\frac{1}{3}$.  As just noted in the previous section, the Lebesgue measure of the ball can be estimated by
\[ |B(y,d)| \approx \prod_{\mu} \frac{r}{(|\mu|r)^\frac{1}{2} + (K r^2)^\frac{1}{3}} \lesssim K^{-\frac{n-k}{3}} D^{- \frac{1}{2}} d^{\frac{k}{2} + \frac{n-k}{3}}.   \]
Likewise, when $N \geq k+1$ and $\beta \ggreater 1$, the inequality \eqref{hesssize1} coupled with Fubini's theorem (integrating first over those directions corresponding to eigenvectors of $H_y$ with ``large'' eigenvalue) gives that
\[ \int_{B(y,d)} \frac{dz}{1 + (\lambda N_z^*[\nabla \Phi(z)])^N} \lesssim C_N \lambda^{-\frac{k}{2}} D^{-\frac{1}{2}} K^{-\frac{n-k}{3}} d^\frac{n-k}{3}. \]

Suppose that $\Phi$ is a $C^\infty$ function whose Newton polyhedron contains only vertices of degree three, that is, there exists a homogeneous cubic polynomial $p$ such that $\Phi(x) - p(x)$ vanishes to order $4$ at the origin.  Suppose further that the rank of the Hessian of $p$ is at least $k$ at every point $x \neq 0$.  It follows that, for $x$ sufficiently small but nonzero, the rank of the Hessian of $\Phi$ will also be at least $k$.  Moreover, for fixed $\Phi$ and $\beta$, if $\delta$ is sufficiently small, then $\rank_{\beta} B(x,\delta r) \geq k$ when $r$ is the distance from $x$ to the origin (here the nonisotropic and isotropic distances are comparable).  Covering $B(0,2^{-j} \lambda^{-1}) \setminus B(0,2^{-j-1} \lambda^{-1})$ by a bounded number of balls with radius comparable to $\delta 2^{-j} \lambda^{-1}$ and summing over $j \geq 0$ as in the previous section gives that, for some constant $C_N$ (independent of the choice of some small vector $\xi \in \R^n$):
\begin{align*} 
\int_{B(0,d)} \frac{dz}{1 + (\lambda N_z^*[\nabla \Phi(z) + \xi])^n} & \\   \lesssim C_N |B(0, \lambda^{-1})| & + C_N \sum_{j=0}^\infty  \lambda^{-\frac{k}{2}} ( \lambda 2^{-j})^\frac{k}{6} K^{-\frac{n-k}{3}} (2^j \lambda^{-1})^\frac{n-k}{3} 
\end{align*}
(since $D$ will be larger than a constant times $(2^{-j} \lambda)^\frac{k}{3}$).  Provided $\frac{k}{6} > \frac{n-k}{3}$, the infinite sum will converge and be bounded above by a constant times $\lambda^{-\frac{n}{3}}$.  By \eqref{sublevel}, it would follow that \eqref{oscillate} satisfies a uniform estimate with decay $\lambda^{-\frac{n}{3}}$ as well.

Therefore, it is natural to ask the following question:  for a generic cubic polynomial $p$, how low can the rank of the Hessian fall at points away from the origin?  An analogous version of this question also arises in the work of Greenleaf, Pramanik, and Tang \cite{gpt2007}, in which they ask about the rank of a generic mixed Hessian.  In that work, they prove what they call a ``rank 1 condition,'' meaning that for a generic mixed Hessian, the rank never falls to zero (except at a trivial point corresponding to the origin).  In the context of the theorem at hand, however, a rank 1 condition is far too weak to obtain optimal estimates for the cubic integrals (even in the original paper, it provides optimal results only for polynomials of very high degree, corresponding to operators with bounded rates of decay in $\lambda$).

Thankfully, one can prove a significantly stronger version of the rank 1 condition.  In fact, the result of the following proposition is that the rank of the Hessian of a generic cubic polynomial $p$ at any point $x \neq 0$ is always greater than the integer part of $n - \sqrt{2n}$ (which is asymptotically far better than even the necessary $\frac{2n}{3}$).  This result is somewhat surprising because the codimension of ``bad'' cubics (for which uniform estimates fail) inside the set of all cubics is much higher than $1$ for large $n$.  In the quadratic case, the best possible uniform estimates hold locally if and only if the determinant of the Hessian is nonzero at some point.  In contrast, if one were to attempt to explicitly characterize the ``good'' set of cubics, it would necessarily need to be described as the set of cubics on which any one of a number of different ``hyperdeterminants'' is nonzero.
\begin{proposition}
There is a dense open subset $U_n \subset {\mathfrak S}^3_n$ such that the Hessian of any $p \in U_n$ evaluated at any $x \neq 0$ in $\R^n$ has rank greater than or equal to the integer part of $n - \sqrt{2n}$.  
\end{proposition}
\begin{proof}
Suppose momentarily that $T_x$ is some family of linear transformations from $\R^m$ to $\R^n$ for $m \geq n$ which depends smoothly on $x = (x_1,\ldots,x_k) \in \R^k$.  Using the standard Euclidean structures on $\R^n$ and $\R^m$, there exists an orthogonal projection $P_R$ on $\R^m$ projecting onto the kernel of $T_x$, and $P_L$ on $\R^n$ projecting onto the kernel of $T^t$.  Now any $T_y$ for $y$ sufficiently near $x$ must have rank at least $r$.  For the rank of $T_y$ to equal $r$, it is necessary and sufficient that
\begin{equation}
 P_L T_y P_R - P_L T_y (I - P_R) T_y^{-1} (I - P_L) T_y P_R = 0, \label{reduce}
\end{equation}
which is proved by the standard row-reduction techniques (note that, while $T_y^{-1}$ is not defined, the product $A := (I - P_R) T_y^{-1} (I - P_L)$ is defined so that $P_R A = A P_L = 0$ and $(I - P_L) T_y (I - P_R) A = (I- P_L)$, etc.)
It follows by the implicit function theorem, then, that the codimension of the set of rank $r$ transformations near $x$ is at least equal to the dimension of the span of the space of operators $P_L ( \partial_{x_j} T_x) P_R$ for $j=1,\ldots,k$.

To apply this observation to the current proposition, an appropriate family of operators must be constructed:  given any nonzero $x \in \R^n$ and any nonzero homogeneous cubic polynomial $p$ in $n$-variables, let $T_{x,p}$ equal the Hessian matrix of $p$ evaluated at $x$.  Suppose $x = (1,0,\ldots,0)$; let $p_{ij}(y) := y_i y_j y_1$.  The Hessian of $p_{ij}$ evaluated at $x$ is a symmetric matrix whose only nonzero entries are the $i,j$ and $j,i$ entries.  Thus the span of all such Hessians is the entire space of symmetric matrices.  Differentiating $T_{x,p}$ in the direction of $p_{ij}$ then, for all $i \leq j$, the resulting matrices must again span.  Finally, an appropriate orthogonal transformation can map any nonzero $x$ to $(1,0,\ldots,0)$, so at any nonzero $x$ and nonzero $p$, the derivatives of $T_{x,p}$ span all symmetric matrices.  If $T_{x,p}$ is rank $r$, then, $P_L = P_R$ (since $T$ is symmetric), and the span of $P_L ( \partial_{ij} T_{x,p}) P_R$ (where $\partial_{ij}$ is differentiation in the direction of $p_{ij}$)  must therefore have dimension $\frac{1}{2} (n-r)(n-r+1)$.  Thus, if one projects the incidence relation
\[ \Lambda_r := \set{(x , p) \in \R^n \setminus \{0\} \times {\mathfrak S}_n^3 \setminus \{0 \} }{\mbox{rank } T_{x,p} \leq r} \]
back onto the space of homogeneous cubics ($T_{x,p} \mapsto p$), the projection of $\Lambda_r$ has codimension at least $\frac{1}{2} (n-r)(n-r+1) - n$ (which is nontrivial provided $r \leq n - \sqrt{2n}$).  The projection of $\Lambda_r$, however, is precisely the set of those polynomials $p$ for which the rank of the Hessian of $p$ is less than or equal to $r$ at some nonzero point.
%
\end{proof}
By virtue of this proposition, the argument presented at the opening of this section will hold generically when $\sqrt{2n} \leq \frac{n}{3}$, giving the condition $n \geq 18$ found in theorem \ref{oscthm2}.  It is also worth noting that the same arguments yield results when $\Phi$ has a nonzero Hessian at the origin, if it is assumed that the kernel of the Hessian has dimension at least a constant times $\sqrt{n}$ (just as in the argument presented, if the kernel is nonempty, then its dimension cannot be too low).

\section{Acknowledgements}

The author would like to thank E. M. Stein, A. Greenleaf, and D. H. Phong for helpful comments on earlier versions of this paper. 

%
%
%
%

\bibliography{mybib}
\end{document}